\documentclass[a4paper,11pt,reqno]{amsart}

\usepackage{amsbsy,amssymb,amsmath,amsthm}
\usepackage{caption,color,graphicx,epsfig}
 \usepackage[a4paper,text={6in,9in},centering]{geometry}
\usepackage[numbers,comma,square,sort&compress]{natbib}
\usepackage{hyperref} 
\usepackage{bbm}
\usepackage{enumitem}

 \setcaptionmargin{0.25in}

\newtheorem{Lemma}{Lemma}[section]

\newtheorem{Proposition}{Proposition}[section]

\newtheorem{Theorem}{Theorem}[section]

\title{Atomic disintegrations for partially hyperbolic diffeomorphisms}

\author[{\tiny{Ale Jan Homburg}}]{Ale Jan Homburg}
\address{KdV Institute for Mathematics, University of Amsterdam, Science park 904, 1098 XH Amsterdam, Netherlands}
\address{Department of Mathematics, VU University Amsterdam, De Boelelaan 1081, 1081 HV Amsterdam, Netherlands}
\email{a.j.homburg@uva.nl}

\begin{document}

\begin{abstract}
Shub \& Wilkinson and Ruelle \& Wilkinson studied a class of volume preserving 
diffeomorphisms on the three dimensional torus
that are stably ergodic. The diffeomorphisms are  partially hyperbolic
and admit an invariant central foliation of circles.
The foliation is not absolutely continuous, in fact, 
Ruelle \& Wilkinson established that  
the disintegration of volume along central leaves is atomic.
We show that in such a class of volume preserving diffeomorphisms  
the disintegration of volume along central leaves is a single delta measure.
We also formulate a general result for conservative three dimensional 
skew product like diffeomorphisms on circle bundles, providing conditions for delta measures as disintegrations
of the smooth invariant measure. 
\\

\noindent MSC 37C05, 37D30
\end{abstract}

\maketitle 

\section{Introduction}


We consider volume preserving perturbations of
the following diffeomorphisms on the three dimensional torus 
$\mathbb{T}^3 = (\mathbb{R}/\mathbb{Z})^3$: 
\begin{align} \label{e:A3}
(x,y,z) &\mapsto (A(x,y) , z),
\end{align}
where $A \in \operatorname{GL} (2,\mathbb{Z})$ is a hyperbolic torus automorphism.

The interest in these systems stems from their role in the study of stable ergodicity.
Indeed, Shub \& Wilkinson \cite{MR1738057} show the existence, 
arbitrarily close to \eqref{e:A3}, of a $C^1$ open set of $C^2$ volume preserving diffeomorphisms 
that are ergodic with respect to volume.
Stable ergodicity has since been shown to occur abundantly in conservative partially hyperbolic diffeomorphisms \citep{pugshu04,burwil10}.

Our interest comes from the phenomenon of Fubini's nightmare \cite{mil97} 
that appears in these diffeomorphisms and is related to non absolutely continuous foliations. 
By classical work on normal hyperbolicity \cite{MR0501173},
perturbations of \eqref{e:A3} admit an invariant center foliation with leaves that are
circles close to $\{ (x,y) = \text{constant} \}$ (which is the invariant center foliation for \eqref{e:A3}).   
The diffeomorphisms studied in \cite{MR1738057} are shown by Ruelle \& Wilkinson \cite{MR1838747} to possess a set of full Lebesgue measure
that intersects almost every circle from the center foliation in $k$ points for some finite integer $k$.
The number $k$ remained unspecified in their result. We will show that the result in \citep{MR1838747} is true with $k=1$.
We thus get robust examples of conservative diffeomorphisms on $\mathbb{T}^3$ with a center foliation of circles and   
an invariant set of full Lebesgue measure that intersects almost every center leaf in a single point.



The theorem below  recalls the results of \citep{MR1738057,MR1838747}.
Note that the center Lyapunov exponent $\lambda^c$  in the formulation of the theorem 
is negative,
the inverse diffeomorphisms possess a positive center Lyapunov exponent
as in \cite{MR1738057}. Also, \cite{MR1738057} takes $A = \left( \begin{array}{cc} 2 & 1 \\ 1 & 1 \end{array}\right)$; the extension to 
arbitrary hyperbolic torus automorphisms is in \cite[Section~7.3.1]{bondiavia05}.                                                                  
                                                              
\begin{Theorem}[\citep{MR1738057,MR1838747}]\label{t:dis1}
In any neighborhood of \eqref{e:A3} there is a $C^1$ open set $\mathcal{U}$ of
volume preserving diffeomorphisms on $\mathbb{T}^3$,
so that for each $C^2$ diffeomorphism $F \in \mathcal{U}$,
\begin{enumerate}[label=(\roman*)]
 \item \label{i:1} $F$ is ergodic with respect to Lebesgue measure;
 \item \label{i:2} there is an invariant center foliation of $C^2$ circles $W^c (p)$, $p \in \mathbb{T}^3$, so that
 for Lebesgue almost all $p$, if $v \in T_p W^c(p)$, then
 \begin{align*}
 \lim_{n \to \infty} \frac{1}{n} \ln |D F^n (p) v| = \lambda^c
 \end{align*}
for some $\lambda^c < 0$;
 \item \label{i:3}
 for some positive integer $k$, 
the disintegrations of Lebesgue measure along center leaves are point measures consisting of $k$ points
with mass $\frac{1}{k}$
(in particular, there is an invariant set of full Lebesgue measure in $\mathbb{T}^3$ that
 intersects almost every center leaf in $k$ points).
\end{enumerate}
\end{Theorem}

The arguments followed by Ruelle \& Wilkinson
involve Pesin theory, in particular
the construction of local unstable manifolds in nonuniformly 
hyperbolic systems. 
With such methods it is not clear how to obtain further information
on the number of atoms $k$. 
As mentioned above, we show that Theorem~\ref{t:dis1} holds with $k=1$.

\begin{Theorem}\label{t:openk=1}
In any neighborhood of $(0,0)$ there is a $C^1$ open set $\mathcal{U}$ of  
volume preserving diffeomorphisms on $\mathbb{T}^3$,
so that each $C^2$ diffeomorphism  $F \in \mathcal{U}$ satisfies properties \ref{i:1}, \ref{i:2} of Theorem~\ref{t:dis1} and furthermore
\begin{enumerate}[label=(\roman*)]
 \item[(iii)] the disintegrations of Lebesgue measure along center leaves are delta measures
(in particular, there is an invariant set of full Lebesgue measure in $\mathbb{T}^3$ for $F$ that
 intersects almost every center leaf in a single point).
\end{enumerate}
\end{Theorem}

The study in \cite{MR1838747} provides a specific two parameter family of diffeomorphisms for which Theorem~\ref{t:dis1}
is shown to hold.
Define $F_{a,b} = (j \circ h)^{-1}$ with
\begin{align}
\nonumber
h(x,y,z) & = (2x + y , x + y , z + x + y + b \sin (2 \pi y)),
\\
\label{e:jh}
j(x,y,z) &= (x +  (1 +  \sqrt{5}) a \cos(2\pi z) ,y + 2a \cos(2\pi z),z).
\end{align} 
For $a,b = 0$, $F_{0,0}$ can be brought to a form  \eqref{e:A3} by a linear coordinate change.
By \cite{MR1838747}, $F_{a,b}$ for small nonzero values of $a,b$ satisfies the conclusions of Theorem~\ref{t:dis1}.
We show that atomic disintegrations with $k=1$ occur within this family.
Specific to the two parameter family of diffeomorphisms $F_{a,b}$ is the existence of a smooth center stable foliation.
This makes the argument more straightforward.

\begin{Lemma}\label{l:affine}
 The center stable foliation of $F_{a,b}$ is the affine foliation with leaves tangent to the planes spanned by $v_0 = (1+\sqrt{5} , 2 , 0)$ and $(0,0,1)$. 
\end{Lemma}

\begin{proof}
Observe that $(1+\sqrt{5} , 2)$ is the unstable eigenvector of the torus automorphism given by
$A = \left( \begin{array}{cc} 2 & 1 \\ 1 & 1  \end{array}\right)$.
The lemma is clear from the observations that $h$ is a skew product diffeomorphism and that $j$ leaves the given affine foliation invariant. 
\end{proof}


\begin{Theorem}\label{t:k=1}
In any neighborhood of $(0,0)$ there  is a set $\Phi$ of positive Lebesgue measure so that for $(a,b) \in \Phi$, 
$F_{a,b}$ satisfies properties \ref{i:1}, \ref{i:2} of Theorem~\ref{t:dis1} and furthermore
\begin{enumerate}[label=(\roman*)]
 \item[(iii)]
 the disintegrations of Lebesgue measure along center leaves of $F_{a,b}$ are delta measures
 (in particular, there is an invariant set of full Lebesgue measure in $\mathbb{T}^3$ for $F_{a,b}$ that
  intersects almost every center leaf in a single point).
\end{enumerate}
\end{Theorem}

%

It should be noted that disintegrations with $k>1$ points do occur for specific diffeomorphisms in any neighborhood of \eqref{e:A3}.
Namely, if $j$ in \eqref{e:jh} is replaced by 
$(x,y,z) \mapsto (x +  (1 +  \sqrt{5}) a \cos(2\pi q z) ,y + 2a \cos(2\pi q z),z)$ for an integer $q$ with $q \ge 2$,
then $F_{a,b}$ satisfies the $\mathbb{Z}_q$-symmetry relation $F_{a,b} (x,y,z + 1/q) = F_{a,b}(x,y,z) + (0,0,1/q)$. 
By a remark due to Katok and contained in Ruelle and Wilkinson's paper, this forces $k$ to be a multiple of $q$.

The method to prove Theorem~\ref{t:dis1} is sufficiently general to treat some other partially hyperbolic systems.
Let $M$ be a compact three dimensional manifold $M$, for which
there exists a circle bundle (a fiber bundle with circles as fibers) $\pi: M \to \mathbb{T}^2$
over the two dimensional torus. 
Let $A$ be an Anosov diffeomorphism on $\mathbb{T}^2$.
We say that a diffeomorphism $G$ on $M$  is a partially hyperbolic skew product over $A$ if $G$ preserves the fibration of the circle bundle, which is the center foliation, and $G$ projects to $A$. The relevance of this definition is underlined by 
 \cite[Theorem 1]{bonwil05} and \cite[Theorem 1]{bon13}; these results 
provide simple topological conditions that guarantee a topological conjugacy to a partially hyperbolic skew product over
an Anosov diffeomorphism. 
We refer to \citep{MR2068774,bondiavia05,rodrodure11} for background and additional information on partially hyperbolic diffeomorphisms.

 Recall that a foliation on a manifold is minimal if all its leaves lie dense in the manifold.

\begin{Theorem}\label{t:general}
Let $A$ be an Anosov diffeomorphism on $\mathbb{T}^2$.
Let $F$ be a partially hyperbolic diffeomorphism, preserving a smooth measure $m$,
that is topologically conjugate to a partially hyperbolic skew product over $A$. 
Assume the following properties.
\begin{enumerate}[label=(\roman*)]
\item $F$ is ergodic with respect to  $m$;
\item $F$ has  a center Lyapunov exponent $\lambda^c < 0$;
\item $F$ has a minimal strong unstable foliation;
\item \label{i:4} $F$ admits a hyperbolic periodic point $P = F^k(P)$ so that
\begin{enumerate}
\item $F^k$ restricted to the periodic center leaf $W^c(P) = F^k (W^c(P))$ is Morse-Smale with a 
unique attracting fixed point $P$ and unique repelling fixed point $Q$;
\item $\lambda^u (Q) \lambda^c (Q) > \lambda^u (P)$ (where $\lambda^u(Q)$ is the strong unstable  eigenvalue of $DF^k (Q)$,
 $\lambda^c(Q)$ is the central  eigenvalue of $DF^k (Q)$ and likewise at $P$).
 \end{enumerate}
\end{enumerate} 
Then the disintegrations of $m$ along center manifolds are delta measures.
\end{Theorem} 

We illustrate Theorem~\ref{t:general} with an example of a partially hyperbolic skew product system from \cite{bonwil05}.  
Start with the map
\[A_\theta (x,y,z) = ( A (x,y) , z + \theta (x,y)  )  \]
on $\mathbb{T}^3$, where 
$A  = \left( \begin{array}{cc}  3 & 2 \\ 1 & 1  \end{array} \right)$ and 
$\theta: \mathbb{T}^2 \to \mathbb{R}$ is a smooth map satisfying 
$\theta (x , y+ \frac{1}{2}) = - \theta (x,y)$.
Consider the action of $\mathbb{Z}_2$ on $\mathbb{T}^3$ induced by 
$\varphi (x,y,z) = (x,y+\frac{1}{2} ,-z)$. The quotient of $\mathbb{T}^3$ by this
$\mathbb{Z}_2$-action is a smooth manifold $M$.
By \cite[Proposition 4.1]{bonwil05}, $A_\theta$ projects to a partially hyperbolic skew product
diffeomorphism $F_\theta: M\to M$. It preserves the smooth measure $m$ that comes from Lebesgue measure on $\mathbb{T}^3$.
The center foliation is a nonorientable circle bundle.

\begin{Proposition}
In any neighborhood of $F_0$ (meaning $F_\theta$ with $\theta (x,y) = 0$), there is a $C^1$ open set $\mathcal{U}$ of diffeomorphisms
on $M$, preserving the smooth measure $m$, so that for each $C^2$ diffeomorphism $F \in \mathcal{U}$,
\begin{enumerate}[label=(\roman*)]
 \item \label{i:pp1} $F$ is ergodic with respect to $m$;
 \item \label{i:pp2} there is an invariant center foliation of $C^2$ circles $W^c (p)$, $p \in M$,
 with  center Lyapunov exponent $\lambda^c \ne 0$;
 \item 
 the disintegrations of $m$ along center leaves are delta measures.
\end{enumerate}
\end{Proposition}

\begin{proof}[Sketch of proof]
Consider the family  $A_{a,b} = j \circ h$ on $\mathbb{T}^3$ with
\begin{align*}
h(x,y,z) &= (3 x + 2 y  , x + y , z +  b \sin (2 \pi y)),
\\
j(x,y,z) &= (x +  (1 + \sqrt{3}) a \cos (2 \pi z) , y +  a \cos (2 \pi z) , z). 
\end{align*}
Note that we recover $A_\theta$ for $\theta (x,y)= 0$ if $a,b=0$.
A direct calculation shows that $A_{a,b}$ is  volume preserving as well as equivariant with respect to the given $\mathbb{Z}_2$-symmetry,
and hence projects to a diffeomorphism on $M$.
One also checks that for small $a,b \ne 0$, there are precisely two hyperbolic fixed points $(0,0,\frac{1}{4})$ and $(0,0,\frac{3}{4})$.
For small values of $(a,b)$, $A_{a,b}$  possesses a fixed center leaf near $\{(x,y)= (0,0) \}$.
This leaf is normally hyperbolic and therefore contains the fixed points $(0,0,\frac{1}{4})$ and $(0,0,\frac{3}{4})$.
The additional eigenvalue conditions from item~\ref{i:4} in Theorem~\ref{t:general} hold since $A_{a,b}$ has a smooth center unstable foliation,
compare Lemma~\ref{l:affine} and the proof of Lemma~\ref{l:onehalf}.   

Let $F_{a,b}$ denote the projected diffeomorphism on $M$.
By Hirsch, Pugh \& Shub \cite{MR0501173}, or \cite[Proposition 4.1]{bonwil05}, $F_{a,b}$ and small perturbations thereof
are topologically conjugate to a partially hyperbolic skew product over $A$.
By \cite{RH^2U}
the set of ergodic partially hyperbolic diffeomorphisms is $C^1$ open and dense. 
Baraviera and Bonatti \cite{barbon03} show how a nonzero center Lyapunov exponent is created through small 
local perturbations (if needed).
A minimal strong unstable foliation is created through an arbitrarily small perturbation with a blender 
as in Lemma~\ref{l:PQopen} below.
All these properties are robust, so that an open set of diffeomorphisms is created
for which 
 the conditions of Theorem~\ref{t:general} hold (take the inverse diffeomorphism in case of a positive 
center Lyapunov exponent). 
Apply Theorem~\ref{t:general}.
\end{proof}


We finish the introduction with the proof strategy for Theorems~\ref{t:openk=1} and \ref{t:k=1} in keywords.
We make use of a 
Markov-like partition of $\mathbb{T}^3$ coming from a Markov partition on leaf space.
On the partition elements there is a projection obtained 
by identifying points on local strong stable manifolds. Composing the diffeomorphism $F$ with the projection defines a factor $F^+$
of $F$.
There is a one-to-one relation between invariant measures of $F^+$ and of $F$, 
we use this to express the disintegration $\mu_p$ of Lebesgue measure on the center leaf through a point $p$
as a limit of pushforwards of disintegrations of projected Lebesgue measure. Combined with the dynamics of $F$ this allows us to conclude that $\mu_p$ is a delta measure, for Lebesgue almost all points $p$. Ingredients from the dynamics we use are, apart from ergodicity with respect to Lebesgue measure, 
a minimal strong unstable foliation 
and a fixed center leaf with Morse-Smale dynamics containing a pair of an attracting and a repelling fixed point.
These dynamical ingredients and the negative center Lyapunov exponent, 
give the existence of a set of center leaves of positive Lebesgue measure on which a large interval is contracted to small intervals under
 iteration.



The careful comments by anonymous referees were a great help to improve the paper.
I am grateful to the referee who pointed out a missing argument in a previous version. 

\section{Proofs of the results on delta measures as disintegrations}

In order to avoid too much jumping between cases, we
will prove Theorems~\ref{t:openk=1} and \ref{t:k=1} and deal with Theorem~\ref{t:general} by 
noting that, apart from notation, it follows from the same arguments.
Write $W^{i} (p)$, $i= s,c,u$, for the strong stable, center or strong unstable manifold containing
$p$. Further, $W^{sc} (p)$ is the center stable manifold and $W^{cu} (p)$ is the center unstable
manifold containing $p$. We will denote Lebesgue measure on $\mathbb{T}^3$ by $\text{vol}$.
   
Minimal strong stable or strong unstable foliations are abundant in the context
of partially hyperbolic diffeomorphisms \citep{MR1954435}.
The following two lemmas make this precise for members of the family $F_{a,b}$ and for diffeomorphisms close to \eqref{e:A3}.


\begin{Lemma}\label{l:PQ}
In any neighborhood of $(0,0)$ there  is a set $\Phi$ of positive measure so that for $F = F_{a,b}$ with $(a,b) \in \Phi$, 
 \ref{i:1}, \ref{i:2} from Theorem~\ref{t:dis1} hold, and
further the following properties:  
 \begin{enumerate}[label=(\roman*)]
 \item \label{i:open1ab} $F$ admits a hyperbolic fixed point $P = P_{a,b}$ so that
 $F$ restricted to the fixed center leaf $W^c(P) = F (W^c(P))$ is Morse-Smale with a 
 unique attracting fixed point $P$ and unique repelling fixed point $Q = Q_{a,b}$;
 \item \label{i:open2ab} the strong unstable and strong stable foliations of $F$ are minimal. 
\end{enumerate}
\end{Lemma}

\begin{proof}
A calculation shows that $F_{a,b}$, for nonzero $a$, has hyperbolic fixed points $(0,0,\frac{1}{4})$ and $(0,0,\frac{3}{4})$.
%
Consider the skew product system $F_{0,b}$ for small $b$ and $a=0$.
Calculate
\[F^{-4}_{0,b} (x,y,z) =
 ( 34x+21y,21x+13y, z + 33x + 21y + b R_4 (x,y))\]
 with
 \[ R_4 (x,y) = \sin (2\pi y) +  \sin (2 \pi (x+y)) +  \sin (2 \pi (3x+2y)) +  \sin (2 \pi (8x+5y)).\]
Note that  $F_{0,b}$ has a period four fiber $\left( \frac{1}{15}, \frac{2}{15}, \mathbb{T}\right)$:
\[
F^{-4}_{0,b} \left(\frac{1}{15},\frac{2}{15},z\right) = \left( \frac{1}{15},\frac{2}{15}, z + bR_4\left(\frac{1}{15},\frac{2}{15}\right)\right) 
\]
with
 \[ R_4 \left(\frac{1}{15},\frac{2}{15}\right) =  \sin\left(\frac{4}{15}\pi\right) + \sin\left(\frac{6}{15}\pi\right) + \sin\left(\frac{14}{15}\pi\right) + \sin\left(\frac{6}{15}\pi\right).\]
Since $R_4 \left( \frac{1}{15},\frac{2}{15}\right) >0$ (all four terms are positive), we find that for a full measure set of values of $b$, $F^{-4}_{0,b}$ has irrational rotation on the period four fiber. 

Treat $F_{a,b}$ as a small perturbation of the family $F_{0,b}$ with $b \ne 0$.
Since the period four fiber of $F_{0,b}$ is normally hyperbolic,  this family possesses a smooth normally hyperbolic period four center fiber $V^c_{a,b}$ near $(\frac{1}{15}, \frac{2}{15}, \mathbb{T})$ \cite{MR0501173}.
For a positive measure set $\Phi$ of parameter values, the rotation number of $F^4_{a,b}$ on $V^c_{a,b}$  is irrational, 
see e.g. \cite[Section I.6]{MR1239171}.
So the strong unstable manifold of a point in $V^c_{a,b}$ is dense in the center unstable manifold of the point. 
Since center unstable manifolds are dense in $\mathbb{T}^3$, this shows that the strong unstable manifold
of any point in $V^c_{a,b}$, $(a,b) \in \Phi$, is dense in $\mathbb{T}^3$.  Compare also \cite[Theorem~12]{MR2366230}.
The same reasoning applies to strong stable manifolds.   
%

The statement that \ref{i:1}, \ref{i:2} from Theorem~\ref{t:dis1} hold is contained in \cite{MR1738057}.
\end{proof}

\begin{Lemma}\label{l:PQopen}
In any neighborhood of \eqref{e:A3} there is a diffeomorphism $F$ for which \ref{i:1}, \ref{i:2} from Theorem~\ref{t:dis1} hold, and
 further the following properties:  
 \begin{enumerate}[label=(\roman*)]
  \item \label{i:open1} $F$ admits a hyperbolic fixed point $P$ so that
  $F$ restricted to the fixed center leaf $W^c(P) = F (W^c(P))$ is Morse-Smale with a 
  unique attracting fixed point $P$ and unique repelling fixed point $Q$; 
 \item \label{i:open2} the strong unstable and strong stable foliations of $F$ are minimal. 
\end{enumerate}
Moreover, these properties are robust.
\end{Lemma}

\begin{proof}
As a consequence of Lemma~\ref{l:PQ}, in and near the family $F_{a,b}$ with $a,b$ small we find examples of diffeomorphisms for which the first item holds. 
For the second item, \citep{MR1954435} discusses 
minimal strong unstable and strong stable foliations in general, not necessarily conservative 
(volume preserving) diffeomorphisms. 
But with the basic tool of blenders available for conservative diffeomorphisms \citep{RH^2TU}, and using a connecting lemma for conservative diffeomorphisms \cite{boncro04},  
their construction can be followed and thus the second item holds. 

For convenience of the reader we spend a few words on clarifying the use of blenders.
%
Start with a diffeomorphism possessing a fixed point $P$ with one dimensional unstable manifold
and a fixed point with two dimensional unstable manifold $Q$, such as \eqref{e:jh}.
A blender associated with $P$ is an open set $V$ near $P$ so that $W^{u}(P)$ intersects each
center stable strip that stretches through $V$ (see the references mentioned above).
In \citep{RH^2TU} it is established that there are arbitrarily small
perturbations of such diffeomorphisms that admit a heterodimensional cycle.
Blenders are found in further arbitrarily small perturbations from here, and 
hence blenders occur arbitrarily close to $F_{a,b}$.

We note that a blender associated with $P$ gives a hyperbolic set, containing a  dense set of periodic points
with one dimensional unstable manifold, close to $P$.
By \cite{boncro04}, 
an arbitrarily small perturbation ensures that
$W^{s}(Q)$ intersects $V$.
Then $W^{cu}(Q) \subset  \overline{W^{u}(P)}$:
high iterates of a small neighborhood $O$ of a point in  $W^{cu}(Q)$ under $F^{-1}$
accumulate onto $W^{s}(Q)$ by the $\lambda$-lemma
and hence contain points accumulating onto $W^{u}(P)$ due to the 
blender associated with $P$.

Since center unstable leaves are dense in $\mathbb{T}^3$ and hence 
$W^{cu}(Q)$ is dense in  $\mathbb{T}^3$, we get that
$W^{u}(P)$ is dense in $\mathbb{T}^3$. 
Since strong unstable manifolds accumulate onto $W^{u}(P)$, all
strong unstable manifolds are dense in $\mathbb{T}^3$, that is, the strong unstable foliation is minimal.
Similarly one obtains a minimal strong stable foliation.
%
\end{proof}


We use a partition of $\mathbb{T}^3$ which is perhaps best 
explained by using 
a topological conjugacy to a skew product system, as in the following result from 
Hirsch, Pugh \& Shub \cite{MR0501173}.

\begin{Proposition}\label{l:goro}
There is a homeomorphism $H$ on $\mathbb{T}^3$ with $H \circ F = G \circ H$,
where $G$ is a skew product diffeomorphism
\begin{align*}
 G(x,y,z) &= (A(x,y) , G_{x,y}(z)),
\end{align*}
for the hyperbolic torus automorphism $A$ and with $z\mapsto G_{x,y}(z)$ a diffeomorphism depending continuously on $(x,y)$.
\end{Proposition}

Take a Markov partition $\mathcal{R} = \{ R_1,\ldots,R_n\}$ for the base dynamics $(x,y) \mapsto A(x,y)$.
Recall that a partition element $R_i$ is a rectangle, bounded by 
segments in local stable and local unstable manifolds. One can bound the diameter of the rectangles by any given $d>0$.
Consider the partition of $\mathbb{T}^3$ with partition elements $R_i \times \mathbb{T}$. 
The image under the topological conjugacy $H^{-1}$ is a partition
$\{S_1,\ldots,S_n\}$ of $\mathbb{T}^3$. 
The conjugacy $H^{-1}$ maps boundaries of $R_i \times \mathbb{T}$ into center stable and center unstable manifolds of $F$, so that 
the boundaries of $S_i$ lie in center stable and center unstable manifolds of $F$.
A partition element $S_i$ is therefore diffeomorphic
to a product of a rectangle and a circle.
Note that the boundaries of the partition elements (and their forward and backward orbits) 
are of zero Lebesgue measure.
For $p$ in the interior of $S_i$, 
we write $W^{s}_{loc}(p)$ for the local strong stable manifold containing $p$
with boundary points in the boundary of a partition element $S_i$.
Likewise other local invariant manifolds 
have their boundary
inside the boundary of a partition element $S_i$.

\begin{Proposition}\label{p:pesin}
There are $R>0$ and a set $\Lambda\subset \mathbb{T}^3$ that is  of positive Lebesgue measure, 
so that 
\begin{enumerate}[label=(\roman*)]
\item \label{i:p1} For $p\in \Lambda$, $\Lambda \cap W^c(p)$ contains an interval $B(p) \subset W^c(p)$
with a length uniformly bounded from below by $R$;
\item \label{i:p2} There are $C>0$, $\ell < 1$ so that 
\begin{equation} \label{e:unicon}
\left| F^n(q)-F^n(r) \right| \le C \ell^n
\end{equation} 
for $q,r$ from $B(p)$.
\end{enumerate}
Moreover, there is a set $\Lambda$ with these properties that is an $s$-saturated set:
 \begin{equation*} 
 \Lambda = \cup_{p\in\Lambda} W^s_{loc} (p).
 \end{equation*}
\end{Proposition}

\begin{proof}
The statements on the existence of a set $\Lambda$ of positive Lebesgue measure so that
items~\ref{i:p1}, \ref{i:p2} hold 
can be found in \cite{MR1838747}. 
The bound \eqref{e:unicon} (possibly with a different contant $C$) 
also holds when one replaces $\Lambda$ by its $s$-saturation $\cup_{p\in\Lambda} W^s_{loc} (p)$.
This is true since the stable holonomy map $h_{p,q}: W^c (p) \to W^c(q)$, defined for $q \in W^{sc}_{loc}(p)$ by 
$h_{p,q}(x) = W^s_{loc} (x) \cap W^c(q)$, 
is uniformly $C^1$ \citep{pugshuwil97,burwil10}. This shows that we may take $\Lambda$ to be an $s$-saturated set.
\end{proof}

The following lemma contains a key argument for the proof of Theorems~\ref{t:openk=1}~and~\ref{t:k=1}.
Its proof uses the above proposition and also relies on minimality of the strong unstable foliation.
We denote 
leaf measure (Lebesgue measure) on center leaves by $\lambda$.
Given a center manifold $W^c (p)$,  $F|_{W^c(p)} \lambda$ denotes the push forward measure
$F|_{W^c(p)} \lambda  (A) = \lambda (F^{-1} (A))$ on $F (W^c (p))$.

\begin{Lemma}\label{l:delta2}
Let $F$ satisfy the properties of Lemma~\ref{l:PQopen}, or let $F= F_{a,b}$ with $(a,b) \in \Phi$.
For Lebesgue almost all $p \in \mathbb{T}^3$, 
$ \{ F^n|_{W^c (F^{-n}(p))} \lambda \}$ contains a delta measure in its limit points
in the weak star topology.
\end{Lemma}

\begin{proof}
For intervals $I$ we write $|I|$ to denote their length, so for intervals $I$
inside center leaves we also write $|I| = \lambda(I)$. 
Fix $\varepsilon >0$.\\

\noindent {\em Step 1}.
Recall from Lemma~\ref{l:PQ} the existence of a center leaf, fixed by $F$, 
containing an attracting fixed point $P$ and a repelling fixed point $Q$.
Note that any closed interval in $W^c (P) \setminus Q$ is contracted under iteration by $F$.
The existence of strong stable and strong unstable foliations near $W^c(P)$ 
shows that a similar contraction occurs on center leaves near $W^c(P)$
as long as iterates remain near $W^c(P)$.
Let $K_0\subset W^{u}_{loc} (P)$ be a fundamental interval with endpoints $k_0, k_1 = F^{-1}(k_0)$. 
Write $K_{n} = F^{-n}(K_0)$. Note that the intervals $K_n$ converge to $P$ as $n \to \infty$. 
Now there is $N \in \mathbb{N}$
so that for $q_{-N} \in W^s_{loc}(K_{N})$, there is $V \subset W^c (q_{-N})$ with both
\[
 |V| > 1-\varepsilon \; \mathrm{ and } \; |F^N (V)| < \varepsilon.
\]
Larger values of $N$ are needed for smaller values of $\varepsilon$.
For use in the following step we note that 
a stronger contraction is obtained (the image $F^N(V)$ can be made smaller) 
when taking $N$ larger.\\

\noindent {\em Step 2}.
The second step in the proof leads to the following statement. 
For any $\varepsilon >0$ there exists a set $\Lambda_N$ 
and an integer $L$, 
so that  for $r \in \Lambda_N$ there is an interval $V \subset W^c(r)$ of length at least $1 - \varepsilon$ so that
for any integer $n \ge L$, 
$f^n (V)$ has length smaller than $\varepsilon$. The set $\Lambda_N$ will be constructed to have positive Lebesgue measure, as shown in the third step.

The following reasoning is illustrated in Figure~\ref{f:expl}.
Let $\Lambda$ be the set of positive Lebesgue measure provided by Proposition~\ref{p:pesin}.
For simplicity we assume $C=1$ in \eqref{e:unicon}.
For $p \in \Lambda$, let $D(p)$ be a closed subinterval of $B(p)$ some distance, say $R/10$, away from the boundary of $B(p)$. 
The strong unstable manifold of $P$ lies dense and in fact iterates of a fundamental interval $K_0$ lie dense in $\mathbb{T}^3$.
We therefore get that for all $p \in \Lambda$, there are a positive integer $M = M(p)$ and a point $q_0 = q_0(p) \in K_0$ with
$F^M(q_0) \in W^s_{loc} (D(p)) \subset W^{sc}_{loc} (p)$.
By replacing $\Lambda$ with a smaller set we get $M$ to be constant.
Namely, write $\tilde{\Lambda}_j \subset \Lambda$ for the set of points $p\in\Lambda$ with $M(p) = j$. At least one of the sets $\tilde{\Lambda}_j$ 
has positive Lebesgue measure. Now replace $\Lambda$ by this $\tilde{\Lambda}_j$ and $M$ will be constant.
Let
\begin{align}\label{e:lambda0}
\Lambda_0 &= \cup_{p\in\Lambda} W^{sc}_{loc} (q_{0}(p))
\end{align}
denote the union of the local center stable manifolds of the points $q_0(p) \in K_0$.

Using the first step, we find $N$ large and $V \subset W^c (q_{-N})$, with $q_{-N} = F^{-N} (q_0)$, 
so that the iterate $F^M$ maps $F^N(V) \subset W^c(q_0)$ into $W^s_{loc}(B(p))$.
By the last sentence of Step 1, we may take an $N$ that works for all $p\in\Lambda$. 
Write $L= N+M$.
Observe that $F^L$ maps  $V \subset W^c(q_{-N})$ into $W^s_{loc}(B(p))$.

Now 
\begin{align*}
\Lambda_N &= \cup_{p\in\Lambda} W^{sc}_{loc} (q_{-N})
\end{align*} 
is the required set.
Note that $\Lambda_N$ is defined as a union of local center stable manifolds.\\
%

\noindent {\em Step 3}.
We prove that $\Lambda_N$ has positive Lebesgue measure.
Its measure equals the measure of
 $F^L ( \cup_{p \in \Lambda} W^{sc}_{loc}(q_{-N}) )$.
For fixed $p$, $F^L ( W^{sc}_{loc}(q_{-N}) )$ is a cylinder inside $W^{sc}_{loc} (F^M (q_0)) = W^{sc}_{loc}(p)$ and hence it
intersects $W^s_{loc}(p)$ in a subinterval. 
Since $\Lambda$ is $s$-saturated, see Proposition~\ref{p:pesin}, it follows that
 $F^L(\Lambda_N) \cap \Lambda$  consists of a subinterval in each local strong stable leaf inside $\Lambda$. 
 Since $L$ is fixed, there exists $c>0$ so that for each $r \in \Lambda_N$, 
 \begin{align}\label{e:>c}
  \left|  F^L (W^s_{loc} (r))    \right| > c.
 \end{align}
 We finish the argument by employing absolute continuity of the strong stable foliation.
 We may write, 
  for a Borel set $A$ contained in a partition element $S_i$ and for a choice of $r \in S_i$,
 \[
 \text{vol}(A) = \int_{W^{cu}_{loc}(r)} \lambda^{s}_p (A \cap W^{s}_{loc} (p)) \, d\nu^{cu} (p),
 \]
 where $\nu^{cu}$ is projected measure of local strong stable manifolds;
 $\nu^{cu} (B) = \text{vol}(\cup_{p\in B} W^s_{loc} (p))$.
By e.g. \cite[Section~8.6]{barpes07},  $\nu^{cu}$ is equivalent to leaf measure on $W^{cu}_{loc}(r)$ and
the conditional measure $\lambda^s_p$ is equivalent to leaf measure on $W^s_{loc} (p)$ 
with density function that is bounded and bounded away from zero.
From this and \eqref{e:>c} we find that
\[
 \text{vol}( F^L (\Lambda_N) \cap \Lambda \cap S_i) = \int_{W^{cu}_{loc}(r)} \lambda^{s}_p ( F^L (\Lambda_N) \cap \Lambda \cap W^{s}_{loc} (p)) \, d\nu^{cu} (p)
 \]
 is positive if $\text{vol}( \Lambda \cap S_i)$ is positive.
Therefore $F^L(\Lambda_N) \cap \Lambda$ and thus $\Lambda_N$ has positive Lebesgue measure.\\

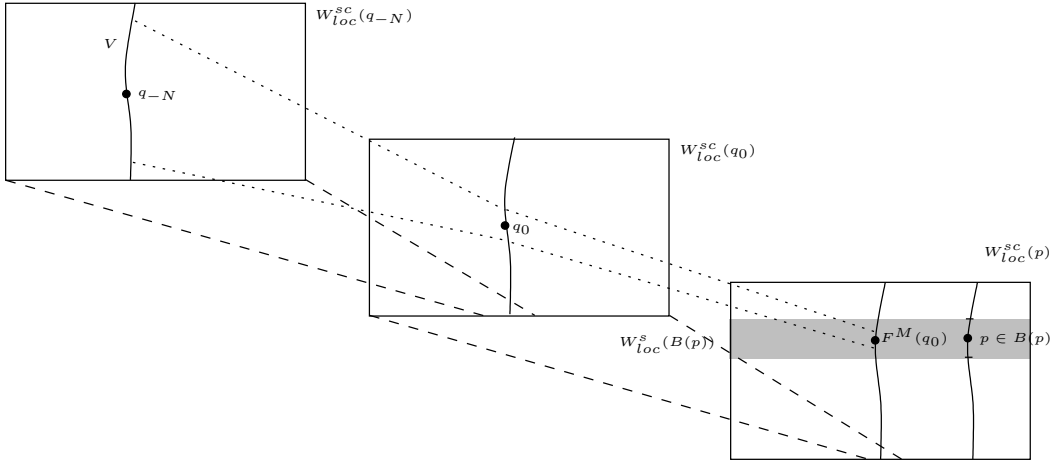
\begin{figure}
\begin{picture}(0,0)%
\includegraphics{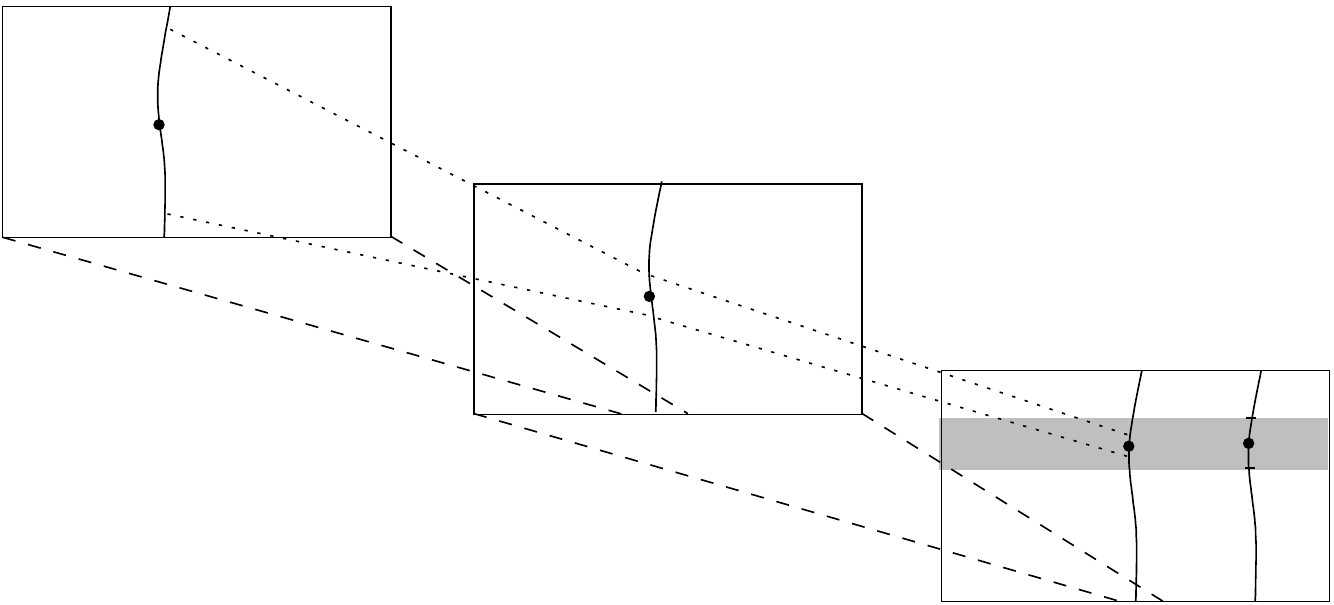}%
\end{picture}%
\setlength{\unitlength}{4144sp}%
\begingroup\makeatletter\ifx\SetFigFont\undefined%
\gdef\SetFigFont#1#2#3#4#5{%
  \reset@font\fontsize{#1}{#2pt}%
  \fontfamily{#3}\fontseries{#4}\fontshape{#5}%
  \selectfont}%
\fi\endgroup%
\begin{picture}(6091,2747)(182,-2728)
\put(5986,-1501){\makebox(0,0)[lb]{\smash{{\SetFigFont{5}{6.0}{\rmdefault}{\mddefault}{\updefault}{\color[rgb]{0,0,0}$W^{sc}_{loc}(p)$}%
}}}}
\put(5374,-2018){\makebox(0,0)[lb]{\smash{{\SetFigFont{5}{6.0}{\rmdefault}{\mddefault}{\updefault}{\color[rgb]{0,0,0}$F^M(q_0)$}%
}}}}
\put(3826,-2041){\makebox(0,0)[lb]{\smash{{\SetFigFont{5}{6.0}{\rmdefault}{\mddefault}{\updefault}{\color[rgb]{0,0,0}$W^{s}_{loc}(B(p))$}%
}}}}
\put(5959,-2013){\makebox(0,0)[lb]{\smash{{\SetFigFont{5}{6.0}{\rmdefault}{\mddefault}{\updefault}{\color[rgb]{0,0,0}$p\in B(p)$}%
}}}}
\put(976,-546){\makebox(0,0)[lb]{\smash{{\SetFigFont{5}{6.0}{\rmdefault}{\mddefault}{\updefault}{\color[rgb]{0,0,0}$q_{-N}$}%
}}}}
\put(773,-263){\makebox(0,0)[lb]{\smash{{\SetFigFont{5}{6.0}{\rmdefault}{\mddefault}{\updefault}{\color[rgb]{0,0,0}$V$}%
}}}}
\put(2027,-86){\makebox(0,0)[lb]{\smash{{\SetFigFont{5}{6.0}{\rmdefault}{\mddefault}{\updefault}{\color[rgb]{0,0,0}$W^{sc}_{loc}(q_{-N})$}%
}}}}
\put(4184,-895){\makebox(0,0)[lb]{\smash{{\SetFigFont{5}{6.0}{\rmdefault}{\mddefault}{\updefault}{\color[rgb]{0,0,0}$W^{sc}_{loc}(q_0)$}%
}}}}
\put(3192,-1346){\makebox(0,0)[lb]{\smash{{\SetFigFont{5}{6.0}{\rmdefault}{\mddefault}{\updefault}{\color[rgb]{0,0,0}$q_0$}%
}}}}
\end{picture}%
\caption[]{An illustration of the proof of Lemma~\ref{l:delta2}: for $\varepsilon>0$ and a given strip $W^s_{loc} (B(p))$, we find a uniformly bounded $L$ so that $F^L$ 
maps a large interval $V \subset W^c (q_{-N})$ of length at least $1-\varepsilon$,
for a $q_{-N} \in K_{N}$, into   $W^s_{loc} (B(p))$.
This uses minimality of the strong unstable foliation. \label{f:expl}}
\end{figure}

We now conclude the proof of the lemma as follows. Take a sequence $\varepsilon_n \to 0$ as $n \to \infty$. 
For each $\varepsilon = \varepsilon_n$, 
one constructs a set $\Lambda_N$.
By ergodicity, there is a set $S_n \subset \mathbb{T}^3$ with $\text{vol} (S_n) = 1$ so that  
$F^{-n} (p)$ intersects the constructed $\Lambda_N$ infinitely often for $p \in S_n$.
The lemma follows for $p \in \cap_n S_n$, noting that $\text{vol} (\cap_n S_n) =1$.
\end{proof}


\begin{Proposition}\label{p:convab}
Let $F$ satisfy the properties of Lemma~\ref{l:PQopen}, or let $F = F_{a,b}$ with $(a,b) \in \Phi$.
Then the disintegrations of Lebesgue measure along center leaves of $F$ are delta measures.
\end{Proposition}

\begin{proof}
Recall the partition $\{ S_1,\ldots, S_n\}$ of $\mathbb{T}^3$ 
and consider $F$ acting on the union $S = \cup_i S_i$ of partition elements.
Note that $F$ acting on $\mathbb{T}^3$ is obtained by gluing partition elements along boundaries.

The proposition is proved by applying \cite[Proposition~3.1]{cra90} (see also \cite[Theorem~1.7.2]{arn98})
that treats relations between invariant measures for endomorphisms and their natural extensions.
These results are formulated for skew product diffeomorphisms and 
translate to our setting by Proposition~\ref{l:goro}.

For a point $p$ from a partition element $S_i$,
write $\pi^{s} (p)$ for its projection along the leaf $W^{s}_{loc} (p)$ 
onto a center unstable side, which we denote by $T_i$, of $S_i$. 
Write $F^+$ for the dynamical system on $T = \cup_i T_i$, obtained by composing $F$ 
with $\pi^{s}$.  Write $\mu^+ = \pi^s \text{vol}$ (i.e. $\mu^+ (A) = \text{vol}((\pi^s)^{-1}(A)$).

\begin{Lemma}
The measure $\mu^+$ 
  is $F^+$-invariant.
\end{Lemma}

\begin{proof}
By the Markov property of the partition we have
$F^{-1} (\pi^{s})^{-1} (A) = (\pi^{s})^{-1} (F^+)^{-1} (A)$,
for Borel sets $A \subset T$.
Hence
\begin{align*}
\mu^+ (A) &= \text{vol}\left( (\pi^{s})^{-1} (A) \right) = \text{vol}\left( F^{-1} (\pi^{s})^{-1} (A) \right) = 
\text{vol}\left(  (\pi^{s})^{-1} (F^+)^{-1} (A) \right) \\ &= 
\mu^+ \left( (F^+)^{-1} (A) \right),
\end{align*}
which expresses $F^+$ invariance of $\mu^+$.
\end{proof}

We have the following properties, implying that $F$ is the natural extension of $F^+$, 
see \cite[Appendix~A]{arn98}:
\begin{enumerate}[label=(\roman*)]
 \item $F^+$ is a factor  of $F$;
 \item With $\mathcal{F}$ the Borel $\sigma$-algebra on $S$, 
$\mathcal{F}^+$ the Borel $\sigma$-algebra on $T$, 
and  $\mathcal{G} = (\pi^{s})^{-1} (\mathcal{F}^+)$, 
we have $\sigma(F^n (\mathcal{G}), n\in \mathbb{N}) =  \mathcal{F} \mod 0$.
Here $\sigma(F^n (\mathcal{G}), n\in \mathbb{N})$
is 
the $\sigma$-algebra generated by $F^n(\mathcal{G})$.
\end{enumerate}
In this context we obtain the following convergence.
Let $\mu_p$ denote the disintegrations of Lebesgue measure along center leaves $W^c (p)$.
%
So, if $\nu^c$ is the measure $\nu^c (A) = \text{vol} (\cup_{p\in A} W^c(p))$ on the leaf space $\mathbb{T}^2$,
we have
\[
\text{vol}(A) = \int \mu_p (A \cap W^c (p)) \, d\nu^c (p)
\]
for Borel sets $A \subset \mathbb{T}^3$.
Considering $\mu^+$ as a measure on $S$ with $\sigma$-algebra $\mathcal{G}$, we also get
disintegrations $\mu^+_p$ on $W^c (p)$ satisfying $\mu^+_p (B) = \mu^+_{\pi^s (p)} (\pi^s (B))$ for Borel sets
$B \subset W^c(p)$.

\begin{Lemma}\label{l:convergence}
For Lebesgue almost all $p\in \mathbb{T}^3 $,
\begin{align*}
 F^{n}|_{W^c (F^{-n}(p))} \mu^+_{F^{-n}(p)} &\to \mu_{p}
\end{align*}
as $n\to \infty$, with convergence in the weak star topology. Moreover, for Lebesgue almost all $p \in \mathbb{T}^3$, $\mu_p$ is a union of $k$ point measures of mass $\frac{1}{k}$ each.
\end{Lemma}

\begin{proof}
Under the homeomorphism $H$ 
that provides the topological conjugacy $H \circ F = G \circ H$ from Proposition~\ref{l:goro},
Lebesgue measure $\text{vol}$ on $\mathbb{T}^3$ is pushed forward to the measure $H \text{vol}$ 
with a marginal $\Omega$ on $\mathbb{T}^2$.
Let $G^+: H(T) \to H(T)$ be given by $G^+ =  H \circ F^+ \circ H^{-1}$. 
Note that $\nu^+  = H \pi^{s} \text{vol}$ is an invariant measure for $G^+$.
Interpret $\nu^+$ as a measure on $H(S) = \cup_i R_i$ with $\sigma$-algebra $H ( \mathcal{G})$.
Now \cite[Proposition~3.1]{cra90} provides convergence of measures 
\begin{align*}
G^n_{A^{-n} (x,y)} \nu^+_{A^{-n} (x,y)} &\to \nu_{x,y},
\end{align*}
in the weak star topology, for $\Omega$-almost all $x,y$.

If $C \subset \mathbb{T}^2$ is a set of full $\Omega$ measure, then $H \text{vol} (C \times \mathbb{T}) = 1$, 
that is, $\text{vol} (H^{-1} (C \times \mathbb{T})) = 1$.
We hence obtain the following statement.
Take Lebesgue measure $\text{vol}$ and 
consider the corresponding invariant
measure $\mu^+ = \pi^{s} \text{vol}$ for $F^+$. 
While $\text{vol}$ is ergodic, by \citep{cra90} also $\mu^+$ is ergodic. 
One finds convergence
\begin{align*}
 F^{n}|_{W^c (F^{-n}(p))} \mu^+_{F^{-n}(p)} &\to \mu_{p},
\end{align*}
for Lebesgue almost all $p \in \mathbb{T}^3$, 
with convergence in the weak star topology. The measures $\mu_p$ are disintegrations of an invariant measure $\mu$ and 
by \cite[Theorem~1.7.2]{arn98} the measures $\mu^+$ and $\mu$ are in one-to-one correspondence so that $\mu$ equals Lebesgue measure.
By \cite{MR1838747}, 
$F^{n}|_{W^c (F^{-n}(p))} \mu^+_{F^{-n}(p)}$ converges to $k$ point measures of mass $\frac{1}{k}$ each,
for Lebesgue almost all $p \in \mathbb{T}^3$.
\end{proof}

We wish to mimic the argument in the proof of Lemma~\ref{l:delta2}, with Lebesgue measure on
center leaves $W^c (q)$ replaced by $\mu^+_q$. 
Let $S_1$ be the partition element of the Markov partition containing the fixed center leaf with the attracting fixed point 
$P$ and the repelling fixed point $Q$. 
For
the fundamental  interval $K_0 \subset W^{u}_{loc}(P)$ introduced in the proof
of Lemma~\ref{l:delta2}, consider the region
\[ 
 V_0 = W^{sc}_{loc} (K_0).
\]
For $n \ge 1$, write  $V_n = F^{-1} (V_{n-1}) \cap S_1$.
Denote by $B_d  (q) \subset W^c(q)$  the interval of diameter $d$ around $q$ inside
$W^c(q)$.
Consider the union of segments $B_d (q)$ over $q \in K_0$ 
 and let $W_0$ be the local strong stable manifolds of this union;
\[
 W_0 = W^s_{loc} \left(  \cup_{q \in K_0}  B_d (q) \right),
\]
see Figure~\ref{f:s1}.
Recall from the proof of Lemma~\ref{l:delta2} the regular set $\Lambda$ of positive Lebesgue measure
and the integers $N,M$.
By taking $d$ depending on $\varepsilon$ small enough, we get 
\begin{align*} 
 W_0 \cap W^{sc}_{loc} (q) &\subset F^{-M}  (W^s_{loc}   ( B(p)))
\end{align*}
for $F^M(q) \in W^s_{loc}(D( p))$ and $p \in \Lambda$.
Write  $W_N = F^{-1} (W_{N-1}) \cap S_1$ for the images under 
$F^{-1}$ inside the partition element $S_1$.
Observe that for large $N$, $V_N \setminus W_N$ is a box inside $V_N$, very thin in the center direction. See again Figure~\ref{f:s1}.

\begin{figure}
\begin{picture}(0,0)%
\includegraphics{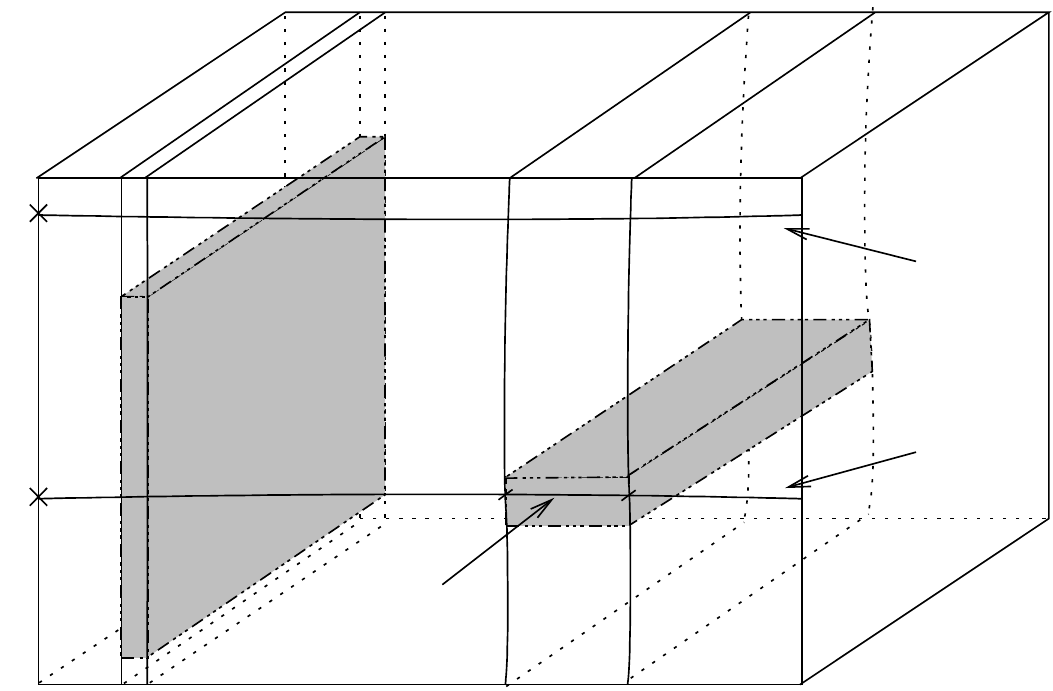}%
\end{picture}%
\setlength{\unitlength}{4144sp}%
\begingroup\makeatletter\ifx\SetFigFont\undefined%
\gdef\SetFigFont#1#2#3#4#5{%
  \reset@font\fontsize{#1}{#2pt}%
  \fontfamily{#3}\fontseries{#4}\fontshape{#5}%
  \selectfont}%
\fi\endgroup%
\begin{picture}(4807,3153)(257,-1582)
\put(2740,-949){\makebox(0,0)[lb]{\smash{{\SetFigFont{6}{7.2}{\rmdefault}{\mddefault}{\updefault}{\color[rgb]{0,0,0}$W_0$}%
}}}}
\put(619,-992){\makebox(0,0)[lb]{\smash{{\SetFigFont{6}{7.2}{\rmdefault}{\mddefault}{\updefault}{\color[rgb]{0,0,0}$W_N$}%
}}}}
\put(2093,-1212){\makebox(0,0)[lb]{\smash{{\SetFigFont{6}{7.2}{\rmdefault}{\mddefault}{\updefault}{\color[rgb]{0,0,0}$K_0$}%
}}}}
\put(4498,323){\makebox(0,0)[lb]{\smash{{\SetFigFont{6}{7.2}{\rmdefault}{\mddefault}{\updefault}{\color[rgb]{0,0,0}$W^u_{loc}(Q)$}%
}}}}
\put(4498,-514){\makebox(0,0)[lb]{\smash{{\SetFigFont{6}{7.2}{\rmdefault}{\mddefault}{\updefault}{\color[rgb]{0,0,0}$W^u_{loc}(P)$}%
}}}}
\put(636,441){\makebox(0,0)[lb]{\smash{{\SetFigFont{6}{7.2}{\rmdefault}{\mddefault}{\updefault}{\color[rgb]{0,0,0}$V_N$}%
}}}}
\put(272,569){\makebox(0,0)[lb]{\smash{{\SetFigFont{6}{7.2}{\rmdefault}{\mddefault}{\updefault}{\color[rgb]{0,0,0}$Q$}%
}}}}
\put(288,-729){\makebox(0,0)[lb]{\smash{{\SetFigFont{6}{7.2}{\rmdefault}{\mddefault}{\updefault}{\color[rgb]{0,0,0}$P$}%
}}}}
\put(2408,140){\makebox(0,0)[lb]{\smash{{\SetFigFont{6}{7.2}{\rmdefault}{\mddefault}{\updefault}{\color[rgb]{0,0,0}$V_0$}%
}}}}
\end{picture}%
\caption[]{This figure illustrates the regions $W_0\subset V_0$ and $W_N \subset V_N$ inside the partition element $S_1$. The vertical direction is the fiber direction: top and bottom sides are identified.
$T_1$ is the front side of $S_1$.
 Note $F^N(V_N) \subset V_0$ and $F^N(W_N) \subset W_0$.\label{f:s1}}
\end{figure}

The following lemma is specific to the family  $F_{a,b}$. 

\begin{Lemma}\label{l:onehalfab}
Let $F = F_{a,b}$ with $(a,b) \in \Phi$.
For $N$ large enough depending on $\varepsilon$, we get for $q \in K_{N}$, 
$\mu^+_q ( W^c(q) \cap W_N   ) >  \frac{1}{2}$.
\end{Lemma}

\begin{proof}
This follows from smoothness of the center stable foliation as stated in Lemma~\ref{l:affine}.
Indeed, 
With $\lambda$ denoting Lebesgue measure on $W^{sc}_{loc}(q)$,
$\lambda (W^{sc}_{loc}(q)\cap (V_N \setminus W_N))$ is uniformly small if $N$ is large. 
Therefore also the projected measure $\mu^+_q ( W^c(q) \cap (V_N \setminus W_N)   )$ is uniformly small if $N$ is large.
\end{proof}


For $F = F_{a,b}$ we can conclude the proof of Theorem~\ref{t:k=1} by the following reasoning.
As a consequence of Lemma~\ref{l:onehalfab},
when replacing
Lebesgue measure with
$\mu^+_q$ on center leaves $W^c(q)$
in the reasoning 
of Lemma~\ref{l:delta2}, we find that for Lebesgue almost all $p \in \mathbb{T}^3$, the limit points of
$F^{n}|_{W^c (F^{-n}(p))} \mu^+_{F^{-n}(p)}$
 contain point measures of mass more than $\frac{1}{2}$.
By Lemma~\ref{l:convergence}, $F^{n}|_{W^c (F^{-n}(p))} \mu^+_{F^{-n}(p)}$ converges to $k$ point measures of mass $\frac{1}{k}$ each. 
So $k\ge 2$ is not possible and $F^{n}|_{W^c (F^{-n}(p))} \mu^+_{F^{-n}(p)}$ converges to a delta measure for Lebesgue almost all $p\in \mathbb{T}^3$.
This proves  Proposition~\ref{p:convab} for $F = F_{a,b}$ and thus 
 Theorem~\ref{t:k=1}.

Smoothness of the center stable foliation as expressed by Lemma~\ref{l:affine} does not hold in general 
and it is not clear whether Lemma~\ref{l:onehalfab} applies in general.
We remark that 
the center stable foliation is absolutely continuous by \cite{viayan10}.

Let $\nu^{sc}$ be the projected measure of local center stable manifolds on $W^u_{loc}(P)$;
$\nu^{sc} (J) = \text{vol} (\cup_{q\in J} W^{sc}_{loc}(q))$.
For a set $A\subset S_1$ we have 
\begin{align*} 
 \text{vol} (A) &= \int  \lambda^{sc}_q (A \cap W^{sc}_{loc}(q)) \, d \nu^{sc} (q).
\end{align*}

Lemma~\ref{l:onehalfab} is replaced by the following. 
The proof of the lemma relies on eigenvalue conditions at the equilibria $P$ and $Q$
that hold for $F_{a,b}$ and perturbations thereof as well as
for the diffeomorphisms considered in \cite[Section~7.3.1]{bondiavia05}.

\begin{Lemma}\label{l:onehalf}
Let $F$ be as in Lemma~\ref{l:PQopen}.
For any $\eta>0$, there is $N>0$ so that there is a set $J \subset K_{N}$ with
$\nu^{sc}(J) > (1-\eta) \nu^{sc}(K_{N})$ and
$\mu^+_q ( W^c(q) \cap W_N   ) >  \frac{1}{2}$ for $q \in J$.
\end{Lemma}

\begin{proof}
Write $\lambda^{s} (P)< \lambda^{c} (P) <\lambda^{u}(P)$ for the eigenvalues of $DF(P)$. Write likewise 
$\lambda^{s} (Q)< \lambda^{c} (Q)< \lambda^{u}(Q)$ for the eigenvalues of $DF(Q)$.
For the system $(j\circ h)^{-1}$ with $j$ and $h$ as in \eqref{e:jh}
we have, because of the affine center stable foliation, 
$\lambda^u(Q) = \lambda^{u} (P)$. The same applies to the diffeomorphisms considered in \cite[Section~7.3.1]{bondiavia05}.
As $ \lambda^c(Q)>1$ we get 
\begin{align}\label{e:bigger}
\lambda^{u}(Q) \lambda^c(Q)  &> \lambda^u(P).
\end{align}
We consider diffeomorphisms close to $(j\circ h)^{-1}$ so that this inequality holds.

We claim that, thanks to \eqref{e:bigger}, 
\begin{align}\label{WNsmall}
\lim_{N\to\infty} \text{vol}(V_N \setminus W_N) / \text{vol}(V_N) &= 0.
\end{align}
For the computations we use local linearizing coordinates near $P$ and $Q$.
As $F$ is a $C^2$ diffeomorphism, there are local $C^1$ diffeomorphisms defined on neighborhoods $O_P$ of $P$ and $O_Q$ of $Q$ in $\mathbb{T}^3$, that transform $F$ into its linearization at $P$ and $Q$ 
\cite{bel73}. The required nonresonance conditions 
$\lambda^c (Q) \ne \lambda^s (Q)   \lambda^u (Q)$ and $\lambda^c (P) \ne \lambda^s (P)   \lambda^u (P)$ 
to apply \cite{bel73} hold since the 
diffeomorphism is conservative and the products
$\lambda^s (Q) \lambda^c (Q) \lambda^u (Q)$ and $\lambda^s (P) \lambda^c (P) \lambda^u (P)$ are therefore equal to $1$.
 By iteration under $F$ we can extend the neighborhoods with linearizing coordinates and we may therefore assume 
$W^c(P) \subset O_P \cup O_Q$. 
There is no loss of generality in assuming that $S_1 \subset O_P \cup O_Q$ and $V_0 \setminus W_0 \subset O_Q$. 

In linearizing coordinates in $O_Q$,
distances in the strong unstable direction get contracted by a factor $1/\lambda^u (Q)$ each iterate under iteration by $F^{-1}$.
This applies to points starting in $V_0\setminus W_0$ that remain in $S_1$ under iteration by $F^{-1}$.
Points in $V_N \setminus W_N$ moreover satisfy an estimate 
$|x_c| \le  C / (\lambda^c(Q))^N $ for some $C>0$. 
It easily follows from these computations that  $\text{vol} (V_N \setminus W_{N}) \sim (\lambda^{u}(Q) \lambda^c(Q))^{-N}$: for some $C>1$,  
 \[ \frac{1}{C} (\lambda^{u}(Q) \lambda^c(Q))^{-N} \le \text{vol} (V_N \setminus W_N) \le C (\lambda^{u}(Q) \lambda^c(Q))^{-N}.\]
Likewise one obtains 
$\text{vol} (O_P \cap V_N) \sim (\lambda^{u}(P))^{-N}$.
By \eqref{e:bigger} we find that 
for large $N$, the volume of $V_N \setminus W_N$  is much smaller than the volume of $O_P \cap V_N$ and hence much smaller 
than the volume of $V_N$. The claim follows. 

By \eqref{WNsmall}, 
it is not possible that the conditional measures $\lambda^{sc}_q$ assign mass $\frac{1}{2}$, or more, to
$(V_N \setminus W_N) \cap W^{sc}_{loc} (q)$ for a nonzero proportion of points $q$  in $K_{N}$ as $N \to \infty$.
Namely, if $\lambda^{sc}_q ((V_N \setminus W_N) \cap W^{sc}_{loc}(q)) \ge \frac{1}{2}$ for $q \in K_N \setminus J$ and  
we pose $\nu^{sc} (K_{N} \setminus J) / \nu^{sc}(K_{N}) \ge \rho$ for some $\rho>0$, then
\begin{align*} \text{vol}(V_N \setminus W_N) &= \int_{K_{N}}  \lambda^{sc}_q ((V_N \setminus W_N) \cap W^{sc}_{loc}(q))\, d \nu^{sc} (q) \\ 
&\ge \int_{K_{N} \setminus J} \frac{1}{2}\, d \nu^{sc} (q) = \frac{1}{2} \nu^{sc}(K_{N} \setminus J) \ge \frac{1}{2} \rho \text{vol}(V_N),
\end{align*}
contradicting \eqref{WNsmall} for $N$ large. Because $\mu_q^+ (A) =  \lambda^{sc}_q  ( \cup_{p\in A} W^s_{loc} (p)  )$ for Borel sets $A\subset W^c(q)$,
the lemma follows.
\end{proof}

Consider $\Lambda_0$ as constructed in the proof of Lemma~\ref{l:delta2}, see \eqref{e:lambda0}, and
write 
 $\Sigma_0 = \Lambda_0 \cap W^{u}_{loc}(P)$.
 We may take $\Lambda_0$ so that the following statement holds, as 
follows from Pesin theory. We take the formulation from \cite[Lemma~6.6]{viayan10}.   
There is a homeomorphism $h: \Sigma_0 \times [-1,1]^2 \to \Lambda_0$, such that
\begin{enumerate}[label=(\roman*)]
 \item $h( \{x_u\} \times [-1,1]^2 ) \subset W^{sc}_{loc}  ( h(x_u,0,0) )$;
 \item \label{i:last2} there is $K>0$ so that for any transversals $\tau_1, \tau_2$ to the center stable foliation, near $K_0$, 
 the center stable foliation induces a holonomy map $h^{sc}$ from $\tau_1 \cap h(x_u \times [-1,1]^2)$ to
$\tau_2 \cap h(x_u \times [-1,1]^2)$ whose Jacobian is bounded by $K$ from above and $1/K$ from below.
\end{enumerate}
As a consequence, there is $\gamma >0$ so that for each transversal $\tau$ to the center stable foliation, near $K_0$,  
\begin{align} \label{e:gamma}
\lambda (\Lambda_0 \cap \tau) &> \gamma.
\end{align}

We claim that for $J \subset K_N$ as in Lemma~\ref{l:onehalf} and $N$ large enough, 
$F^N  (W^{sc}_{loc}(J))$ intersects $\Lambda_0$ in a set of positive Lebesgue measure.
This follows by combining Lemma~\ref{l:onehalf} and \eqref{e:gamma}. 
Namely, 
take a smooth foliation $\mathcal{G}$ of $S_1$ with curves transversal to the local center stable manifolds.
For a measurable set $A \subset S_1$, we can write $ \text{vol} (A) =  \int_{W^{sc}_{loc} (P)}  \lambda (\mathcal{G}_q \cap A ) \,dm(q)$  
for a smooth measure $m$.
By Lemma~\ref{l:onehalf}, $\text{vol} (   W^{sc}_{loc}(J)) / \text{vol} (V_N) > t$ for some $t$ close to one, if $N$ is large.
Write
\[
 \frac{\text{vol} (   W^{sc}_{loc}(J)) }{ \text{vol} (V_N) } =  \frac{ \text{vol} (   W^{sc}_{loc}(J) \cap W_N) }{ \text{vol} (V_N) } + 
 \frac{\text{vol} (   W^{sc}_{loc}(J) \cap (V_N \setminus W_N)) }{ \text{vol} (V_N) }
\]
and observe that the second term on the right hand side goes to zero as $N \to \infty$ by \eqref{WNsmall}.
Hence also $\text{vol} (   W^{sc}_{loc}(J) \cap W_N) / \text{vol} (V_N) > t$ for some $t$ close to one, if $N$ is large.
From this and 
\[ 
\text{vol} (W^{sc}_{loc} (J) \cap W_N) = \int_{W^{sc}_{loc} (P)}  \lambda (\mathcal{G}_q \cap W^{sc}_{loc} (J) \cap W_N)\, dm(q)
\] 
we find that if $N$ is large, 
$\lambda (\mathcal{G}_q \cap W^{sc}_{loc} (J)) / \lambda (\mathcal{G}_q \cap V_N )$ is close to one for some 
$q$ with $\mathcal{G}_q \cap V_N \subset W_N$. 

By bounded distortion \cite[Lemma~3.3]{MR1749677}, with $\lambda (\mathcal{G}_q \cap W^{sc}_{loc} (J)) / \lambda (\mathcal{G}_q \cap V_N )$ 
close to one, also  
$\lambda (F^N (\mathcal{G}_q \cap   W^{sc}_{loc}(J))) / \lambda(F^N(\mathcal{G}_q   \cap V_N))$ is close to one.
(Bounded distortion of $F^N$ on $\mathcal{G}_q$ means there is $C>0$ so that
\[
 \frac{1}{C} \le \frac{| DF^N (q_1) e^{u} |}{| DF^N(q_2)  e^{u} |} \le C,
\]
$q_1,q_2 \in \mathcal{G}_q$, uniformly in $N$,
where $e^{u}$ is a unit tangent vector to $\mathcal{G}_q$.
Consequently, iterating under $F^N$ does not change too much relative length of sets.)
By \eqref{e:gamma},  $F^N (\mathcal{G}_q \cap   W^{sc}_{loc}(J))$
has nonempty intersection, in fact with positive Lebesgue measure, with $\Lambda_0 \cap F^N(\mathcal{G}_q)$, for large enough $N$. 
By item~\ref{i:last2} above, this shows the claim.

The remainder of the proof again follows the  arguments right after Lemma~\ref{l:onehalfab},
 with a smaller set $\Lambda$ still of positive Lebesgue measure
($\Lambda_0$ being replaced by $F^N  (W^{sc}_{loc}(J) \cap \Lambda_0)$).
This concludes the proof of Theorem~\ref{t:openk=1}.
\end{proof}

\def\cprime{$'$}

\end{document}